\documentclass[11pt,reqno]{amsart}
\usepackage{amsmath,amssymb,amsthm,amsfonts,enumerate,hyperref}
\usepackage{graphicx}
\usepackage{float}
\theoremstyle{plain}
\newtheorem{theorem}{Theorem}[section]
\newtheorem{proposition}[theorem]{Proposition}
\newtheorem{cor}[theorem]{Corollary}

\theoremstyle{definition}

\theoremstyle{remark}
\newtheorem{remark}[theorem]{Remark}

\numberwithin{equation}{section}

\begin{document}
\title{Bounding the first invariant eigenvalue of toric K\"ahler manifolds}
\begin{abstract}
We generalise a theorem of Engman and Abreu--Freitas on the first invariant eigenvalue of non-negatively curved $S^{1}$-invariant metrics on $\mathbb{CP}^{1}$ to general toric K\"ahler metrics with non-negative scalar curvature. In particular, a simple upper bound of the first non-zero invariant eigenvalue for such metrics on complex projective space $\mathbb{C}P^{n}$ is exhibited. We derive an analogous bound in the case when the metric is extremal and a detailed study is made of the accuracy of the bound in the case of Calabi's extremal metrics on $\mathbb{C}P^{2}\sharp -\mathbb{C}P^{2}$.
\end{abstract}
\author{Stuart J. Hall}
\address{Department of Applied Computing, University of Buckingham, Hunter St., Buckingham, MK18 1G, U.K.} 
\email{stuart.hall@buckingham.ac.uk}
\author{Thomas Murphy}
\address{Department of Mathematics, California State University Fullerton, 800 N. State College Blvd., Fullerton, CA 92831, USA.}
\email{tmurphy@fullerton.edu}
\maketitle
\section{Introduction}

In this note we investigate the properties of the first \emph{invariant} non-zero eigenvalue of the ordinary Laplacian for toric K\"ahler metrics on a closed manifold $M$. Given a particular toric K\"ahler metric, we will denote this quantity $\lambda_{1}^{\mathbb{T}}$ to distinguish it from the first non-zero eigenvalue $\lambda_{1}$. It is trivial to see that ${\lambda_{1}\leq \lambda_{1}^{\mathbb{T}}}$. The investigation into this quantity began with the works of Engman \cite{E1}, \cite{E2} and Abreu--Freitas \cite{AF} who studied $\mathbb{S}^{1}$-invariant metrics on $\mathbb{S}^{2}$.  In particular Abreu--Freitas proved that, amongst $\mathbb{S}^{1}$-invariant metrics with a fixed volume, $\lambda_{1}^{\mathbb{T}}$ can take any value in $(0,\infty)$.  This is in stark constrast to $\lambda_{1}$ which satisfies, for any Riemannian metric $g$,
$$\lambda_{1} \leq \frac{8\pi}{vol(\mathbb{S}^{2},g)}$$ 
due to a celebrated result of Hersch \cite{H}. If one assumes that the metric satisfies certain geometric properties then one has more control of the quantity $\lambda_{1}^{\mathbb{T}}$.  Under the assumption that $g$ has non-negative Gauss curvature, both Engman and Abreu--Freitas proved that $\lambda_{1}^{\mathbb{T}}$ is bounded (again over the set of metrics with fixed volume). Our main theorem is a generalisation of this result to higher dimensional compact toric K\"ahler metrics. The hypothesis on the Gauss curvature becomes an assumption about the scalar curvature of a toric K\"ahler metric $\omega$.  The assumption about volume can be interpreted as fixing the K\"ahler class ${[\omega]\in H^{2}(M;\mathbb{R})}$ of the metric.  
\begin{theorem}\label{T1}
Let $(M,\omega)$ be a compact toric K\"ahler metric with non-negative scalar curvature.  Then $\lambda_{1}^{\mathbb{T}}$ is bounded above by a quantity that only depends upon the cohomology class ${[\omega]}$.
\end{theorem}
The dependence of the bound on the cohomology class manifests itself as a bound depending upon the moment polytope associated to the torus action. The moment polytope only depends upon the K\"ahler class $[\omega]$. In the course of the proof of Theorem \ref{T1} it becomes clear that it is possible to state the bound in terms of the integral (calculated with a particular measure) of certain polynomials over the boundary of the moment polytope.  In the case of $\mathbb{CP}^{n}$ these integrals are very easy to compute as the moment polytope is a simplex.
\begin{cor}\label{C1}
Let $\omega$ be a toric K\"ahler metric on $\mathbb{CP}^{n}$ with $[\omega] = c_{1}(\mathbb{CP}^{n})$. If $\omega$ has non-negative scalar curvature then 
$$\lambda_{1}^{\mathbb{T}}\leq n+2.$$
\end{cor}
\begin{remark}
 The Abreu--Freitas bound in \cite{AF} for $\mathbb{S}^{1}$-invariant metrics on $\mathbb{S}^{2}$ is  
$$ \lambda_{1}^{\mathbb{T}} <\frac{1}{2}\xi_{1}^{2}\approx 2.89, $$
where $\xi_{1}$ is the first zero of the Bessel function. Hence the bound in Corollary \ref{C1} cannot in general be sharp. The action of $\mathbb{S}^{1}\cong U(1)$ on $\mathbb{S}^{2}$ is a cohomogeneity one action.  One might thus be lead to consider a smaller class of metrics on $\mathbb{CP}^{n}$ invariant under a cohmogeneity one action by $U(n)$. We leave the investigation to the extent to which the bound in Corollary \ref{C1} can be sharpened for the future. 
\end{remark}

Another very natural condition involving the scalar curvature is to require that the metric $g$ is \emph{extremal}. This means that the gradient of the scalar curvature is a real holomorphic vector field. Clearly this is satisfied by K\"ahler metrics with constant scalar curvature (so called cscK metrics). Here we prove: 
\begin{theorem}\label{T2}
Let $(M,\omega)$ be an extremal toric K\"ahler metric. Then $\lambda_{1}^{\mathbb{T}}$ is bounded above by a quantity that only depends upon the cohomology class ${[\omega]}$.
\end{theorem}
The proofs follow the ideas used in \cite{HM1} and \cite{HM2} to investigate the spectrum of the Laplacian of a variety of Einstein metrics with large symmetry groups. The first ingredient is to use the framework developed by Guillemin \cite{G} and Abreu \cite{A} to give coordinates in which various geometric quantities are easy to compute. Secondly we employ an important integration-by-parts formula due to Donaldson \cite{D} (Equation \ref{IBP}).

Finally let us mention the very recent work of Legendre-Sena-Dias \cite{LSD}. They prove that for any toric K\"ahler class $[\omega]$, $\lambda_{1}^{\mathbb{T}}$ can take any value in $(0,\infty)$. This generalises the Abreu-Frietas theorem on $\mathbb{S}^{2}$. They also prove that $\lambda_{1}$ is bounded by a quantity that only depends upon $[\omega]$. 
\bigskip

\noindent \emph{Acknowledgements} We wish to thank Robert Haslhofer for productive conversations. He  suggested investigating Donaldon's integration-by-parts formula to tackle the related problem of calculating the first nonzero eigenvalue of the Laplacian for the CLW metric, using ideas in \cite{HM14}.  We also thank Rosa Sena-Dias and Eveline Legendre for sharing their preprint with us. 

\section{The proof of Theorems \ref{T1} and \ref{T2}}
We begin by recalling the key ideas involved in the Guillemin--Abreu theory of toric K\"ahler metrics.  Associated to any Hamiltonian action of an $n$-torus $\mathbb{T}^{n}$ on a K\"ahler manifold $(M^{2n}, \omega, J)$ ($2n$ is the real dimensions of $M$) is a compact convex polytope ${P\subset\mathbb{R}^{n}}$. The polytope $P$ can always be described as the intersection of the hyperplanes defined by ${\psi_{k}(x)>0}$ where $\psi_{k}(x)$, $k=1,2,...,r$ are certain affine linear functions. Crucially the polytope $P$ depends only on $[\omega]$. There is a dense open set  $M^{\circ}\subset M$ such that $M^{\circ} \cong P^{\circ} \times \mathbb{T}^{n}$. On this open set one can write the metric $\omega$ in the coordinates corresponding to the obvious ones on $P^{\circ}\times \mathbb{T}$. The metric $g=\omega(J\cdot,\cdot)$ can be written as
$$g = u_{ij}dx_{i}dx_{j}+u^{ij}d\theta_{i}d\theta_{j},$$
where $u_{ij}$ is the Euclidean Hessian of a convex function ${u:P^{\circ}\rightarrow\mathbb{R}}$ and $u^{ij}$ is the inverse of this matrix. The function $u$ is referred to as \emph{the symplectic potential}.  A key result of Guillemin \cite{G} is that the coordinate sigularity the metric develops on the boundary is prescribed by the polytope $P$.  He proved that one can always find a smooth function ${f:P\rightarrow\mathbb{R}}$ such that
\begin{equation}\label{symppot}
u(x_{1},..,x_{n}) = \frac{1}{2}\left(\sum_{k=1}^{k=r}\psi_{k}(x)\log (\psi_{k}(x))\right)+f(x).
\end{equation}
Abreu \cite{A} \cite{A2} pioneered the application of this representation of toric K\"ahler metrics to problems involving metrics of special curvature leading to the elegant formula
\begin{equation}\label{Abreueqn}
S=-u^{ij}_{ij}.
\end{equation}   
Building on this work, Donaldson proved an integration-by-parts result as part of his broader program considering the existence theory for solutions of Abreu's equation (\ref{Abreueqn}) where one prescribes the scalar curvature $S$ \cite{D}. In order to state it we need to introduce the idea of the \emph{integral Lesbegue measure}.  The fundamental region of the ordinary lattice $\mathbb{Z}^{n} \subset \mathbb{R}^{n}$ has measure 1 using the usual Lesbegue measure.  A rational linear subspace $H^{d}\subset\mathbb{R}^{n}$ defines a sublattice $H\cap \mathbb{Z}^{n} \cong \mathbb{Z}^{d}$ acting on $H$ by translation.  One then defines a fundamental region of $H$ to be the orbit space of this action.  The integral Lesbegue measure $\sigma$ is then the Lesbegue measure normalised so that this region has measure 1. Translating this measure defines a measure on any affine subsapce parallel to $H$. As an example, the hypotenuse of the triangle with vertices $(0,0), (0,1)$ and $(1,0)$ has length 1 rather than $\sqrt{2}$ with respect to this measure. The classification of moment polytopes of toric varieties due to Delzant \cite{TD} (the admissable polytopes are known as Delzant polytopes) means that the faces are always affine translates of rational linear subspaces and so each face can be given the integral Lesbegue measure.

\begin{proposition}[Donaldson's Integration-by-parts formula, Lemma 3.3.5 in \cite{D}]
Let $u$ be symplectic potential of the form \ref{symppot} for the polytope $P$ and let ${F\in C^{\infty}(P)}$. Then
\begin{equation}\label{IBP}
\int_{P}u^{ij}F_{ij} d\mu= \int_{\partial P}2F d\sigma-\int_{P}SF d\mu. 
\end{equation}
where $S$ is the scalar curvature of the metric defined by $u$, $\mu$ is the usual Lesbegue measure on $\mathbb{R}^{n}$ and $\sigma$ is the integral Lesbegue measure on the boundary $\partial P$.
\end{proposition}
\begin{remark}
In the way it is stated in \cite{D}, there is no factor of $2$ in boundary integral.  This is because Donaldson is using a formulation of the theory where the singular part of the symplectic potential (\ref{symppot}) occurs without the factor of $1/2$.  This is the same as a formulation with the torus fibres have volume $(4\pi)^{n}$ as opposed to the more usual $(2\pi)^{n}$. 
\end{remark}
The proof of the Theorem \ref{T1} follows almost immediately.
\begin{proof}(of Theorem \ref{T1})

If one takes 
$$F(x_{1}, x_{2},...,x_{n}) = \frac{1}{2}\left(\sum_{i=1}^{i=n}b_{i}x_{i}\right)^{2}$$ then
$$\int_{P} u^{ij}F_{ij}d\mu = b_{i}b_{j}\int u^{ij} d\mu \leq \int_{\partial P} 2F d\sigma,$$ 
where the inequality follows from Equation (\ref{IBP}) and the assumption of non-negative scalar curvature. Let $\widetilde{x}_{i} = x_{i}+c_{i}$ so that
$$\int_{P}\widetilde{x}_{i} d\mu =0.$$
Then 
$$\int_{P} u^{ij}F_{ij}d\mu = b_{i}b_{j}\int u^{ij}d\mu = \int_{P}|\nabla \phi|^{2}d\mu,$$
where 
$$\phi(x_{1}, x_{2},...,x_{n}) = \sum_{i=1}^{i=n}b_{i}\widetilde{x}_{i}.$$
Hence
$$\lambda_{1}^{\mathbb{T}} \leq \frac{\|\nabla \phi\|^{2}}{\|\phi\|^{2}}\leq \inf 
\left\{ \frac{\int_{\partial P}\left(\sum_{i=1}^{i=n}b_{i}x_{i}\right)^{2} d\sigma}{\int_{P} \left(\sum_{i=1}^{i=n}b_{i}\widetilde{x}_{i}\right)^{2}d\mu} \ : \ (b_{1},b_{2},...,b_{n})\in \mathbb{R}^{n}\backslash \{0\}\right\}.$$
Clearly the quantity on the righthand side of the above inequality only depends upon the polytope data and the theorem follows.
\end{proof}
Before giving the proof of Corollary \ref{C1} we need a technical result on how to integrate polynomials over simplices.  The calculation is essentially combinatorial and given by the following:
\begin{proposition}[Brion \cite{B} and Lasserre--Avrachenko \cite{LA}]\label{simpInt}
Let $P$ be the $d$-dimensional simplex generated as the convex hull of the $(d+1)$ affinely independent vertices $s_{1},s_{2},...,s_{d+1}$ in $\mathbb{R}^{n}$.  Let ${\Phi:\mathbb{R}^{n}\rightarrow\mathbb{R}}$ be a linear function.  Then
\begin{equation}
\int_{P}\Phi^{q} \ d\sigma = vol(P,d\sigma)\frac{d!q!}{(q+d)!}\sum_{k\in \mathbb{N}^{d+1}, |k|=q}\Phi(s_{1})^{k_{1}}\cdot\cdot\cdot \Phi(s_{d+1})^{k_{d+1}},
\end{equation}
where $\sigma$ is the integral Lesbegue measure and $q\in\mathbb{N}$.
\end{proposition}

\begin{proof}(of Corollary \ref{C1})

The polytope in this case can be chosen as the $n$-simplex in $\mathbb{R}^{n}$ with vertices 
$$s_{i}  = (-1,-1,...n,-1,...-1) \textrm{ and } s_{n+1} = (-1,-1,...,-1). $$ 
Considering the notation used in the proof of Thereom \ref{T1}, we will compute with ${F(x)=x_{1}}$. Given the symmetry of the situation, it turns out this achieves the infimum in the bound. It is straightforward to compute $vol(P,d\sigma) =(n+1)^{n}/n!$ and so
$$\int_{P}x_{1}d\sigma = n!\frac{(n+1)^{n}}{n!}\frac{1}{(n+1)!}\sum_{k\in \mathbb{N}^{n+1}, |k|=1}n^{k_{1}}\cdot(-1)^{k_{2}}\cdot...\cdot(-1)^{k_{n+1}}= 0.$$
The $L^{2}$-norm of $x_{1}$ is given by
$$\|x_{1}\|^{2} = \int_{P}x_{1}^{2} \ d\sigma= n!\frac{(n+1)^{n}}{n!}\frac{2!}{(n+2)!}\sum_{k\in \mathbb{N}^{n+1}, |k|=2}n^{k_{1}}\cdot(-1)^{k_{2}}\cdot...\cdot(-1)^{k_{n+1}}.$$
In this case the sum is calculated over indices of the form either containing a single 2 or two 1s and so
$$\sum_{k\in \mathbb{N}^{n+1}, |k|=2}n^{k_{1}}\cdot(-1)^{k_{2}}\cdot...\cdot(-1)^{k_{n+1}}  = \frac{n(n+1)}{2}.$$
Hence
$$\|x_{1}\|^{2} =\frac{(n+1)^{n}}{n!}\cdot\frac{n}{(n+2)}.$$
In order to compute $\int_{\partial P}x_{1}^{2}d\sigma$ define $\widehat{P}_{i}$ to be the $(n-1)$-dimensional simplex generated by the vertices of $P$ not including $s_{i}$. Then
$$\int_{\partial P} x_{1}^{2} \ d\sigma= \sum_{i=1}^{i=n+1}\int_{\widehat{P}_{i}}x_{1}^{2} \ d\sigma.$$
Calculation yields
$$\int_{\widehat{P}_{i}}x_{1}^{2}d\sigma = vol(\widehat{P}_{i},d\sigma)\frac{2!(n-1)!}{(n+1)!}\left(\frac{n(n+1)}{2}\right),$$
for all $0\leq i \leq n+1$. It is again straightforward to compute $${vol(\widehat{P}_{i},d\sigma) = (n+1)^{n-1}/(n-1)!}$$ and so
$$\int_{\partial P} x_{1}^{2}d\sigma = \frac{n(n+1)^{n+1}}{(n+1)!}.$$
The result now follows.
\end{proof}

Finally, we prove Theorem \ref{T2}.
\begin{proof}(of Theorem \ref{T2})
The extremal condition is equivalent to requiring that 
$$S(x) = a_{0}+\sum_{i=1}^{i=n}a_{i}x_{i}$$
for  $a_{0}, a_{1},...,a_{n} \in \mathbb{R}^{n}$. As pointed out by Donaldson, by considering affine $F$ in Equation (\ref{IBP}) we get $n+1$ linear equations for the $a_{i}$ which we can solve in terms of quantities involving only the polytope data.  The proof now follows in a similar way to that of Theorem \ref{T1}.
\end{proof}
\section{Estimates for extremal metrics on toric surfaces}
\subsection{$\mathbb{CP}^{1}\times\mathbb{CP}^{1}$}
This is the simplest example of a toric surface with non-Einstein extremal metrics.  In this case there is a one parameter family of constant scalar curvature metrics.  The family is are formed by taking the product of the Fubini--Study metric on each of the $\mathbb{CP}^{1}$ factors with each factor scaled by the parameter. More precisely, we take the moment polytope of a metric (normalised to have volume $16\pi^{2}$)  to be the rectangle defined by the linear functions
$$\psi_{1}(x) = a+x_{1}, \ \psi_{2}(x) = a-x_{1}, \ \psi_{3}(x) = a^{-1}+x_{2} \textrm{ and } \psi_{4}(x) = a^{-1}-x_{2}$$
where we can take $a\in [1,\infty)$. The metric in this case is then $g_{a} = ag_{FS}\oplus a^{-1}g_{FS}$ where $g_{FS}$ is the Fubini-Study metric on $\mathbb{CP}^{1}$ with volume $4\pi$. The spectrum of $\Delta_{g_{FS}}$ is given by
$$\sigma(\Delta_{g_{FS}}) = \{k(1+k), k\in \mathbb{N}_{0} \}$$
and so
$$\sigma(g_{a}) = \{a^{-1}k(1+k)+al(1+l), k,l \in \mathbb{N}_{0}\}.$$
The full spectrum is the same as the invariant spectrum in this case as each eigenvalue admits an invariant eigenfunction. Hence $\lambda_{1}^{\mathbb{T}}=2a^{-1}$.

\subsection{$\mathbb{CP}^{2}\sharp-\mathbb{CP}^{2}$}

As in the case of $\mathbb{CP}^{1}\times\mathbb{CP}^{1}$, the surface $\mathbb{CP}^{2}\sharp-\mathbb{CP}^{2}$ admits a one-parameter family of extremal metrics due to Calabi \cite{C}. Again the parameter $a$ really represents the underlying cohomology class of the metric. The metrics were described in the Guillemin--Abreu framework by Abreu \cite{A}. However, we will use the description given in \cite{BHJM}. None of the metrics have constant scalar curvature.  The polytope in this case can be described as the trapezium (trapezoid) defined by the linear functions
$$\psi_{1}(x) = 1+x_{1}, \ \psi_{2}(x) = 1+x_{2}, \ \psi_{3}(x) = a+x_{1}+x_{2} \textrm{ and } \psi_{4}(x) = 1-x_{1}-x_{2},$$
where $a\in (-1,2)$.  Taking $a=1$ yields a metric in the first Chern class $c_{1}(M)$ with volume $16\pi^{2}$ and as in the $\mathbb{CP}^{1}\times\mathbb{CP}^{1}$ case, we shall use this as the normalised volume. If $g(a)$ denotes the extremal metric on $\mathbb{CP}^{2}\sharp-\mathbb{CP}^{2}$ in the cohomology class corresponding to the polytope defined by the above linear functions, then the normalised one parameter family is given by
\begin{equation}\label{normeq}
g_{a} = \frac{2\sqrt{2}}{\sqrt{(a+1)(5-a)}}g(a).
\end{equation}
The extremal metrics $g(a)$ are invariant under the cohomogeneity one action of $U(2)$ and so the extremal equation reduces to an ODE which can be solved explicitly. Using the description in \cite{BHJM}, the Euclidean Hessian of $U(2)$-invariant symplectic potentials can be written as
$$ D^{2}u = \frac{1}{2}\left[\begin{array}{cc} \frac{1}{x_{1}+1}+\frac{z^{-1}(t)-1}{2+t} & \frac{z^{-1}(t)-1}{2+t} \\ \frac{z^{-1}(t)-1}{2+t} & \frac{1}{x_{2}+1}+\frac{z^{-1}(t)-1}{2+t} \end{array} \right], $$
where $t=x_{1}+x_{2}$ and $z(t) \in C^{\infty}([-a,1])$. The scalar curvature of the extremal metrics $g(a)$  can be calculated in the manner described in the previous section.  They are given by 
$${S(x_{1},x_{2}) = \alpha(x_{1}+x_{2})+\beta},$$ 
where
$$\alpha = \frac{48(2-a)}{(a+1)(a^{2}-16a+37)} \textrm{ and } \beta = \frac{12(4a-3a^{2}+13)}{(a+1)(a^{2}-16a+37)}.$$ 
The extremal equation is then reduced to the following second order linear ODE for $z$:
\begin{equation*}
z''+\frac{4}{(2+t)}z'-\frac{2(1-z)}{(t+2)^{2}}+\frac{\alpha t}{2(2+t)}+\frac{\beta}{2(2+t)}=0,
\end{equation*}
with boundary conditions $z(1)=z(-a)=0$. This has solution 
\begin{equation*}
z(t) = \frac{(t-1)(a+t)(a^{2}t - 4a^{2} + 2at^{2} + 10at + 36a - 4t^{2} - 33t - 74)}{(t+2)^{2}(a+1)(a^{2}-16a+37)}.
\end{equation*}
There is currently no known closed form for the spectrum of the extremal metrics $g_{a}$ but we can use the procedure described in the proof of Theorem \ref{T2} to compute an explicit upper bound for $\lambda_{1}^{\mathbb{T}}$. If we let
$$\widetilde{x}_{i} = x_{i}-\frac{(2-a)^{2}}{3(5-a)},$$
then  
$$\int_{P}\widetilde{x}_{i} d\mu = 0.$$
Hence we can compute the Rayleigh quotient of the function ${\phi = b_{1}\widetilde{x}_{1}+b_{2}\widetilde{x}_{2}}$ by using the function
$$F =  b_{1}^{2}\frac{x_{1}^{2}}{2}+b_{1}b_{2}x_{1}x_{2}+b_{2}^{2}\frac{x_{2}^{2}}{2}$$
in Equation \ref{IBP}. As the metrics are all invariant under the $\mathbb{Z}_{2}$ action swapping $x_{1}$ and $x_{2}$, we can assume $b_{1}=1$. A lengthy calculation yields (for the unnormalised extremal metric $g(a)$)
$$ \frac{\|\nabla \phi\|^{2}}{\|\phi\|^{2}} = \frac{A(1+b^{2}_{2})+b_{2}B}{C(1+b_{2}^{2})+b_{2}D},$$
where
$$A = \frac{(a + 1)(a^{4} - 14a^{3} + 132a^{2} - 590a + 883)}{10(a^{2} - 16a + 37)},$$
$$B = -\frac{(a + 1)(7a^{4} - 188a^{3} + 1284a^{2} - 3860a + 4381)}{30(a^{2} - 16a + 37)},$$
$$C =  -\frac{(a + 1)(a^{4} - 14a^{3} + 60a^{2} - 158a + 253)}{36(a - 5)},$$
and
$$D = \frac{(a + 1)(a^{4} - 14a^{3} + 114a^{2} - 374a + 469)}{36(a - 5)}.$$
The function $$f(x) = \frac{A(1+x^{2})+Bx}{C(1+x^{2})+Dx}$$ is extremised at $x=\pm1$. If ${AD-BC<0}$, $f$ has a maximum at ${x=1}$ and minimum at $x=-1$, if ${AD-BC>0}$ then the maximum is at $x=-1$ and minimum at $x=1$. This is reflected in the geometry of the metrics which, as mentioned, are invariant under a $\mathbb{Z}_{2}$ action.  This means that the space of smooth functions decomposes as ${C^{\infty}(M) = C^{+}(M)\oplus C^{-}(M)}$ where $C^{+}(M)$ and $C^{-}(M)$ are the spaces of smooth $\mathbb{Z}_{2}$-invariant and anti-invariant functions respectively.  As the Laplacian preserves both of these subspaces the smallest invariant eigenvalue must correspond to an eigenfunction in one (or possibly both) of these spaces. In the case at hand $AD-BC$ is equal to
$$ -\frac{(a - 2)(a^{2} - 7a + 19)(2a^{4} - 85a^{3} + 777a^{2} - 2233a + 1763)(a + 1)^{3}}{540(a - 5)(a^{2} - 16a + 37)},$$
which is less than zero for $a\in (-1,a_{c})$ and greater than zero for $a \in (a_{c},2)$ where $a_{c}\approx 1.2877$. Hence we arrive at
\begin{theorem}\label{Text}
Let $g_{a}$ be the extremal metric on $\mathbb{CP}^{2}\sharp-\mathbb{CP}^{2}$ described by (\ref{normeq}). Then
$$
\lambda_{1}^{\mathbb{T}} \leq \frac{\sqrt{2(a+1)}(13a^{4} - 272a^{3} + 2076a^{2} - 7400a + 9679)}{10\sqrt{(5-a)}(a^2 - 4a + 13)(a^{2} - 16a + 37)} \textrm{ if } -1<a<a_{c},
$$
and 
$$ \lambda_{1}^{\mathbb{T}} \leq -\frac{3\sqrt{2}(5-a)^{3/2}(a^{3} - 105a^{2} + 597a - 917)}{10\sqrt{(a + 1)}(a^{2} - 16a + 37)^{2}} \textrm{ if } a_{c}\leq a<2.
$$
Here $a_c$ is the unique zero of the polynomial 
$$p(a) = 2a^{4} - 85a^{3} + 777a^{2} - 2233a + 1763,$$
in the range $a\in [-1,2]$.
\end{theorem}
This theorem suggests that for small values of the parameter $a$, it is the $\mathbb{Z}_{2}$-anti-invariant functions that contain the minimising eigenfunction but this switches at some point $a_{c}\approx 1.2877$ and then the invariant functions contain the minimising eigenfunction. We can plot the bound as a function of $a$, the result of this is contained in Figure 1.
As we have an explicit form for each extremal metric, we are able to investigate exactly how accurate the bound in Theorem \ref{Text} is. To do this we use the Rayleigh-Ritz approximation where our set of test functions is
$$T = \{1,x_{1}, x_{2}, x_{1}^{2}, x_{2}^{2}, x_{1}x_{2}\} .$$
Labelling $\phi_{1}(x) = x_{1}$, $\phi_{2}(x)=x_{2}$,..., $\phi_{6}(x) = x_{1}x_{2}$, we compute the two $ 6 \times 6$ matrices   
$$M_{ij} = \langle \phi_{i},\phi_{j}\rangle_{L^{2}} \textrm{ and } \widetilde{M}_{ij} = \langle \nabla\phi_{i},\nabla\phi_{j}\rangle_{L^{2}} .$$
The Rayleigh-Ritz approximation is then given by the smallest positive eigenvalue of the matrix $M^{-1}\widetilde{M}$.  We plot this for 100 equally spaced values of $a$ in Figure 2.  As the plot shows, the behaviour of this quantity is very similar to that of the upper bound in Theorem \ref{Text} with the same switching from the smallest eigenfunction being $\mathbb{Z}_{2}$-anti-invariant to $\mathbb{Z}_{2}$-invariant at $a \approx 1.2877$. To show just how close the bound is to this quantity we plot the difference of the bound and the Rayleigh-Ritz approximation in Figure 3.\\
\\
It is also interesting to note that informally at least, as $a\rightarrow 2$, $${(\mathbb{C}P^{2}\sharp -\mathbb{C}P^{2},g_{a}) \rightarrow (\mathbb{CP}^{2},\gamma g_{FS})}$$ where $\gamma= 2\sqrt{2}/3$. It is well-known that $\lambda_{1}= \lambda_{1}^{\mathbb{T}} = 2$ for $g_{FS}$ and so we see that the absolute upper bound for $\lambda_{1}^{\mathbb{T}}$ in the family of extremal metrics is that of the rescaled Fubini-Study metric $3\sqrt{2}/2 \approx 2.12$. Hence in both the $\mathbb{C}P^{1}\times \mathbb{C}P^{1}$ and ${\mathbb{C}P^{2}\sharp -\mathbb{C}P^{2}}$ cases, the value of $\lambda_{1\mathbb{T}}$ for the one parameter family of extremal metrics is bounded above absolutely by the eigenvalue of the K\"ahler-Einstein metric in the family.\\
\\
There are further toric surfaces formed by blowing up points on $\mathbb{C}P^{2}$ which admit extremal metrics as well as higher dimensional examples.  One difficulty in studying these manifolds is that the precise existence results are much more difficult to state. For example, not every K\"ahler class on the surface $\mathbb{C}P^{2}\sharp-2\mathbb{C}P^{2}$ admits an extremal metric. Nevertheless it should be possible to derive analogues of Theorem \ref{Text} quite easily. Testing the accuracy of the bounds will be more challenging as very few of the metrics are known in a closed form.  There are some approaches to numerically approximating such metrics \cite{BD} \cite{HM14} that would allow the calculation of the Rayleigh-Ritz matrices in simple cases.  The authors hope to discuss this in separate work.

\begin{figure}[H]
\centering
\includegraphics[width=\linewidth, height=212pt]{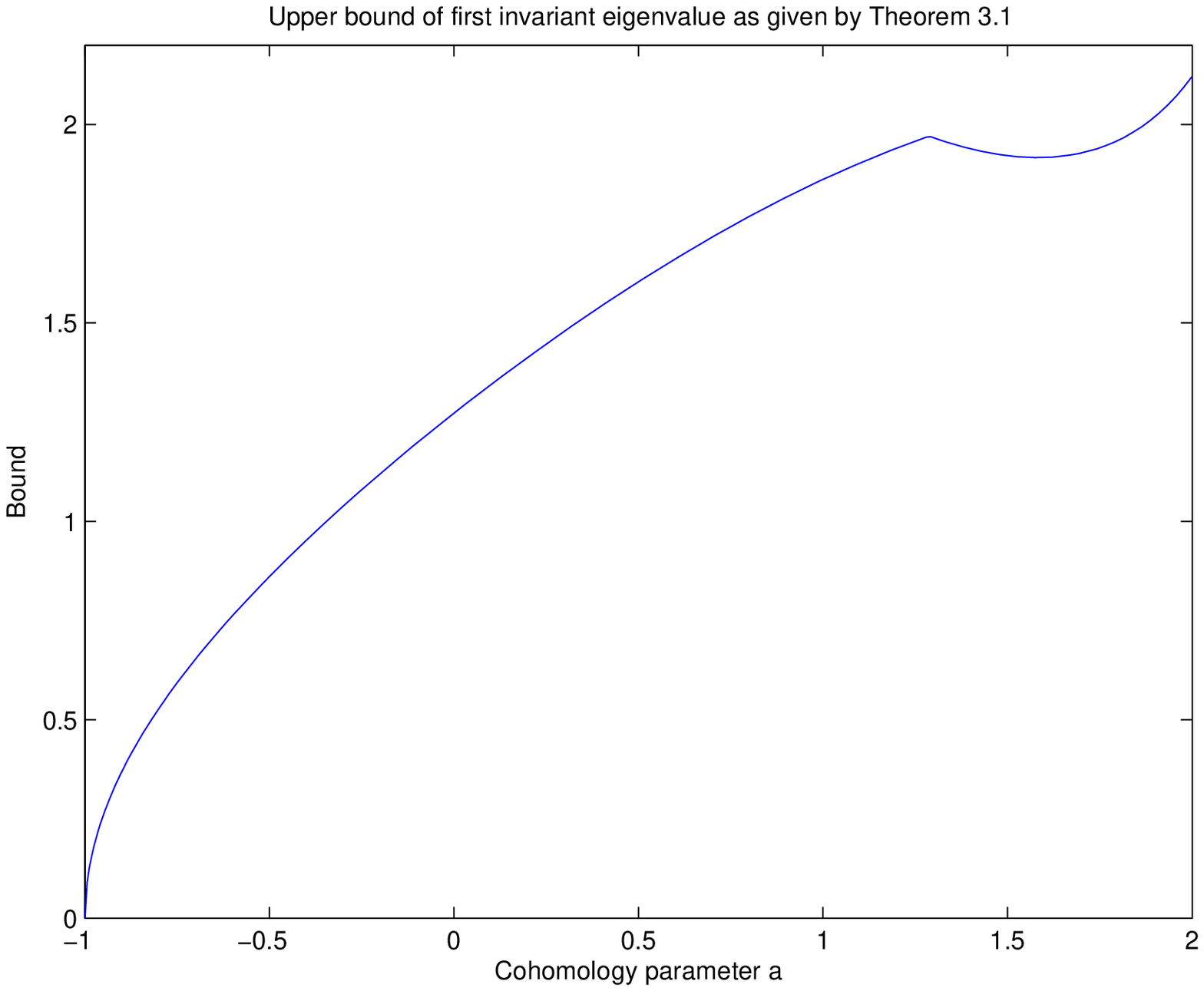}
\begin{small}
\caption{}
\end{small}
\end{figure}
\begin{figure}[H]
\centering
\includegraphics[width=\linewidth, height=212pt]{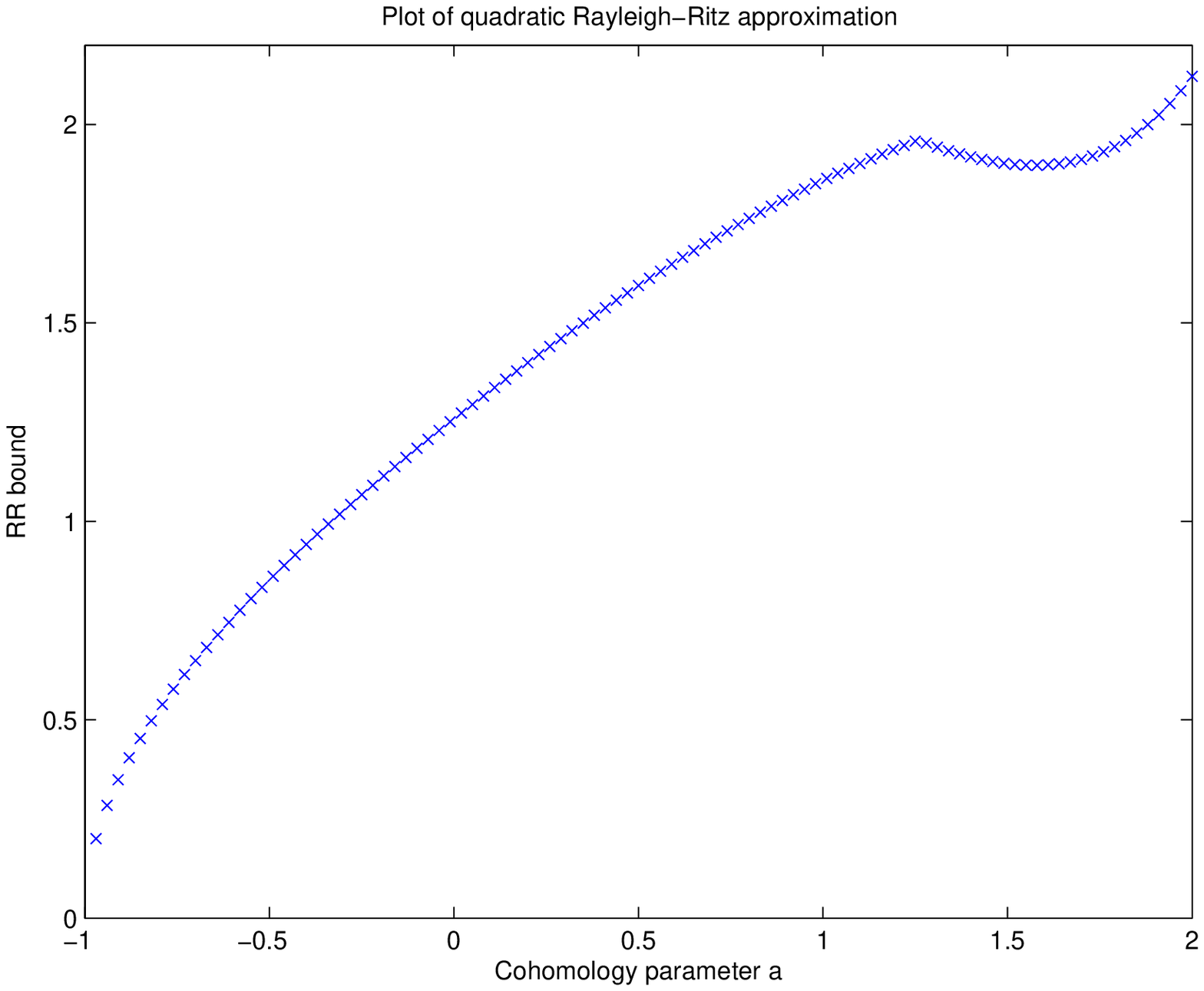}
\begin{small}
\caption{}
\end{small}
\end{figure}
\begin{figure}[H]
\centering
\includegraphics[width=\linewidth, height=212pt]{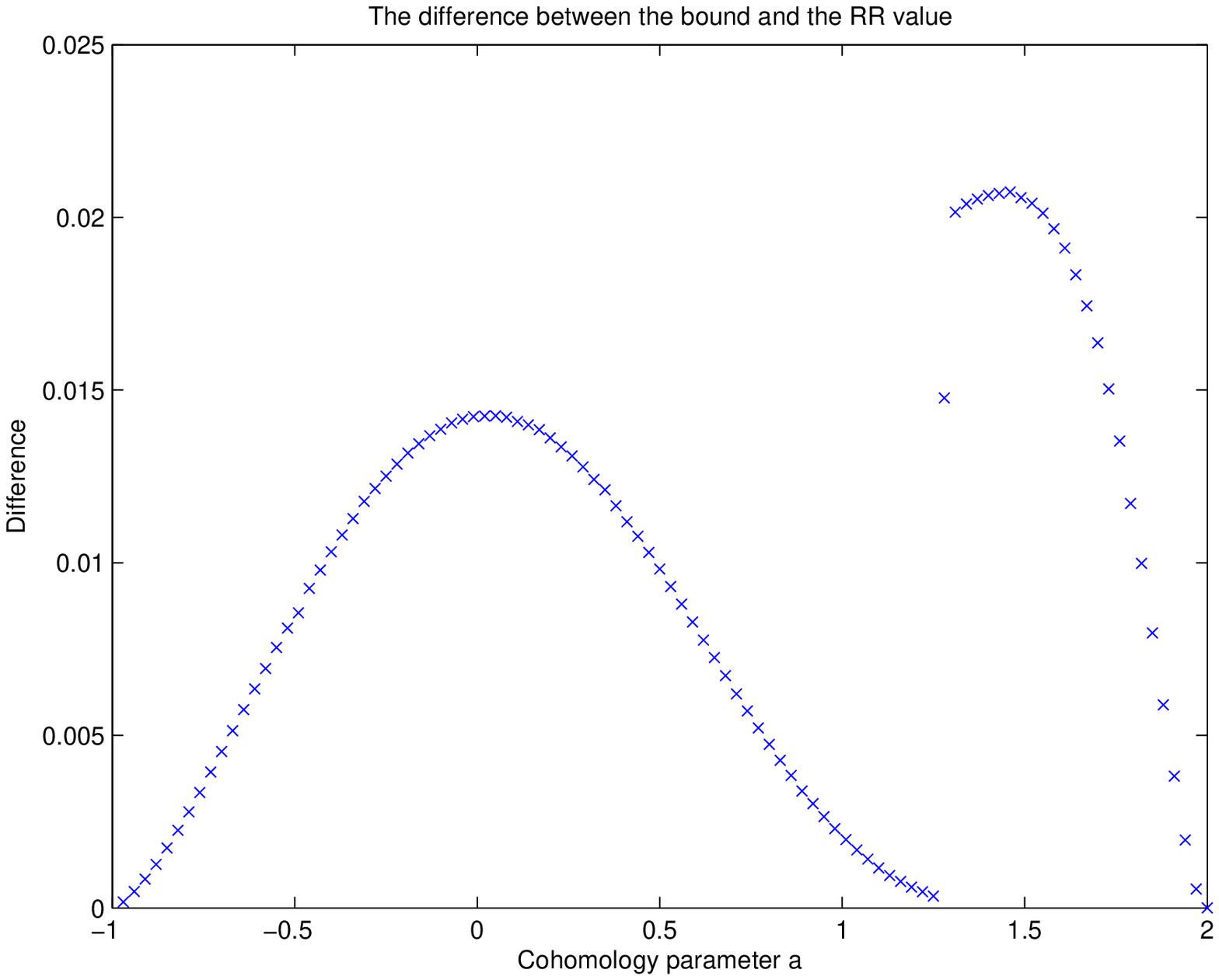}
\begin{small}
\caption{}
\end{small}
\end{figure}

\end{document}